\documentclass[11pt]{amsart}

\usepackage{amssymb}
\usepackage{amsthm}
\usepackage{amsmath}
\usepackage{amsbsy}
\usepackage[all]{xy}
\usepackage{bm}
\usepackage{hyperref}
\usepackage{tikz}
\usepackage{array}
\usepackage{colortbl,xcolor}
\usepackage{ytableau}
\usepackage{hhline}
\allowdisplaybreaks
\usepackage[margin=1.5in]{geometry}

\newtheorem{theorem}{Theorem}[section]

\newtheorem{corollary}[theorem]{Corollary}

\newtheorem{lemma}[theorem]{Lemma}
\newtheorem{question}[theorem]{Question}
\newtheorem{problem}[theorem]{Problem}

\newcommand{\cupdot}{\mathbin{\mathaccent\cdot\cup}}

\begin{document}

\title[]{Generalized difference sets and\\ autocorrelation integrals} \keywords{Generalized difference set, autocorrelation, additive combinatorics.}
\subjclass[2010]{11B13, 11P70, 26D15, 42A85.}

\author[]{Noah Kravitz}
\address[]{Grace Hopper College, Zoom University at Yale, New Haven, CT 06510, USA}
\email{noah.kravitz@yale.edu}

\maketitle
\vspace{-0.8cm}
\begin{abstract}
In 2010, Cilleruelo, Ruzsa, and Vinuesa established a surprising connection between the maximum possible size of a generalized Sidon set in the first $N$ natural numbers and the optimal constant in an ``analogous'' problem concerning nonnegative-valued functions on $[0,1]$ with autoconvolution integral uniformly bounded above.  Answering a recent question of Barnard and Steinerberger, we prove the corresponding dual result about the minimum size of a so-called generalized difference set that covers the first $N$ natural numbers and the optimal constant in an analogous problem concerning nonnegative-valued functions on $\mathbb{R}$ with autocorrelation integral bounded below on $[0,1]$.  These results show that the correspondence of Cilleruelo, Ruzsa, and Vinuesa is representative of a more general phenomenon relating discrete problems in additive combinatorics to questions in the continuous world.
\end{abstract}

\section{Introduction}

\subsection{A tale of four problems}
Consider the following two pairs of problems.  First, a pair of additive combinatorics problems about subsets of the integers (where $[N]=\{1,2,\ldots, N\}$):
\begin{itemize}
\item[(1a)] For natural numbers $g,N$, find the maximum size of a set $A \subseteq [N]$ such that every $m \in \mathbb{Z}$ has at most $g$ solutions in $A$ to the equation $m=a_i+a_j$.  %(Such a set $A$ is called a \emph{$g$-Sidon set for $[N]$}.)
\item[(1b)] For natural numbers $g,N$, find the minimum size of a set $A \subset \mathbb{Z}$ such that every $m \in [N]$ has at least $g$ solutions in $A$ to the equation $m=a_i-a_j$.  %(Such a set $A$ is called a \emph{$g$-difference set for $[N]$}.)
\end{itemize}
Next, a pair of analysis problems about functions on the real line:
\begin{itemize}
\item[(2a)] Find the maximum $L^1$ norm of a function $f: [0,1] \to \mathbb{R}_{\geq 0}$ such that every $x \in \mathbb{R}$ satisfies $\int_{\mathbb{R}} f(t)f(x-t) \, dt \leq 1$.
\item[(2b)] Find the minimum $L^1$ norm of a function $f: \mathbb{R} \to \mathbb{R}_{\geq 0}$ such that every $x \in [0,1]$ satisfies $\int_{\mathbb{R}} f(t)f(x+t) \, dt \geq 1$.
\end{itemize}
Cilleruelo, Ruzsa, and Vinuesa showed in their groundbreaking 2010 paper \cite{cilleruelo} that the optimal constants in Problems (1a) and (2a) are closely related.  Barnard and Steinerberger \cite{barnard} asked whether there is an analogous correspondence between the optimal constants in Problems (1b) and (2b).  The main objective of this paper is to answer this ``dual'' question in the affirmative.

\subsection{Background and definitions}
We now introduce our objects of study in a precise way and survey previous work in the area.  Given a subset $A$ of an abelian group $G$, define the representation counting functions
$$q_A(x)=|\{ (a_1,a_2) \in A \times A: x=a_1+a_2 \}|$$
and
$$r_A(x)=|\{ (a_1,a_2) \in A \times A: x=a_1-a_2 \}|.$$
In this language, we say that $A \subseteq G$ is a \emph{$g$-Sidon set} if $q_A(x) \leq g$ for all $x \in G$, and  we say that $A \subseteq G$ is a \emph{$g$-difference set} if $r_A(x) \geq g$ for all $x \in G$.  In the special case where $G=\mathbb{Z}$ (as discussed above), we are interested in slightly different objects: we say that $A \subseteq [1,N]$ is a \emph{$g$-Sidon set for $[N]$} if $q_A(m) \leq g$ for all $m \in \mathbb{Z}$, and we say that $A \subseteq \mathbb{Z}$ is a \emph{$g$-difference set for $[N]$} if $r_A(m) \geq g$ for all $m \in [N]$.  We remark that a $2$-Sidon set for $[N]$ specializes to the ordinary notion of a Sidon set contained in the first $N$ natural numbers and that a $1$-difference set coincides with the classical notion of a difference set.

It is natural to ask about the extremal sizes of these generalized Sidon and difference sets.  We make the following definitions:
\begin{align*}
&\alpha_g(G)=\max\{|A|: A \subseteq G \text{ is a $g$-Sidon set}\};\\
&\beta_g(N)=\max\{|A|: A \text{ is a $g$-Sidon set for $[N]$}\};\\
&\gamma_g(G)=\min\{|A|: A \subseteq G \text{ is a $g$-difference set}\};\\
&\eta_g(N)=\min\{|A|: A \text{ is a $g$-difference set for $[N]$}\}.
\end{align*}
(Note that $\gamma_g(G)$ exists whenever $g \leq |G|<\infty$ and that $\eta_g(N) \leq 2g<\infty$.)  We record the trivial bounds
$$\alpha_g(G) \leq \sqrt{g|G|}, \quad \beta_g(N) \leq \sqrt{2gN}, \quad \gamma_g(G) \geq \sqrt{g|G|}, \quad \eta_g(N) \geq \sqrt{2gN}.$$
The authors of \cite{cilleruelo} are concerned with $\alpha_g$ and $\beta_g$, and we will be concerned mostly with $\gamma_g$ and $\eta_g$.

Obtaining nontrivial bounds on these quantities, particularly $\beta_2$ (classical Sidon sets in $\mathbb{Z}$), has been the object of considerable interest.  For more background on generalized Sidon sets and applications, we refer the reader to the description in \cite{cilleruelo2, cilleruelo} and the references therein.  Old bounds on difference sets, starting in the 1940's, include the work of R\'{e}dei and R\'{e}nyi \cite{redei}, Leech \cite{leech}, and Golay \cite{golay}.  More recent advances in this area often appear in the context of connections between generalized difference sets and other areas of combinatorics, such as block designs \cite{bose, sumner}, symmetric intersecting families \cite{ellis}, and graceful labelings of graphs \cite{graham}.

We now introduce the other leading characters of this paper, this time from the continuous world.  Recall that the \emph{convolution} integral of $f,h: \mathbb{R} \to \mathbb{R}$ is given by
$$(f*h)(x)=\int_{t \in \mathbb{R}} f(t) h(x-t) \, dt$$
and the \emph{correlation} integral is given by
$$(f \star h)(x)=\int_{t \in \mathbb{R}}f(t)h(x+t) \, dt.$$
When $f=h$, we will speak of the \emph{autoconvolution} and \emph{autocorrelation} integrals of $f$.  We may define the family of all nonnegative functions supported on $[0,1]$ whose autoconvolution integral is everywhere at most $1$ to be
$$\mathcal{E}=\{f: [0,1] \to \mathbb{R}_{\geq 0} \text{ such that $(f*f)(x) \leq 1$ for all $x \in \mathbb{R}$}\}.$$
Similarly, we may define the family of all nonnegative functions on $\mathbb{R}$ whose autocorrelation integral is at least $1$ on all of $[0,1]$ to be
$$\mathcal{F}=\{f: \mathbb{R} \to \mathbb{R}_{\geq 0} \text{ such that $(f\star f)(x) \geq 1$ for all $x \in [0,1]$}\}.$$
We now define the constants
$$\sigma=\sup_{f \in \mathcal{E}} \int_{[0,1]} f(x) \, dx \quad \text{and} \quad \tau=\inf_{f \in \mathcal{F}} \int_\mathbb{R} f(x) \, dx.$$
We have the trivial bounds $\sigma \leq \sqrt{2}$ and $\tau \geq 1$; determining the exact values of $\sigma$ and $\tau$ is of intrinsic interest and appears to be a very difficult problem.

We are finally ready to connect all of these pieces.  The following theorem of Cilleruelo, Ruzsa, and Vinuesa \cite{cilleruelo} shows exactly how $\beta_g(N)$ and $\sigma$ are related.

\begin{theorem}[\cite{cilleruelo}]\label{thm:sidon}
We have the equalities
$$\lim_{g \to \infty} \liminf_{N \to \infty} \frac{\beta_g(N)}{\sqrt{gN}}=\sigma=\lim_{g \to \infty} \limsup_{N \to \infty} \frac{\beta_g(N)}{\sqrt{gN}}.$$
\end{theorem}
Our main result is the analogous statement for $\eta_g(N)$ and $\tau$.
\begin{theorem}[Main Theorem]\label{thm:main}
We have the equalities
$$\lim_{g \to \infty} \liminf_{N \to \infty} \frac{\eta_g(N)}{\sqrt{gN}}=\tau=\lim_{g \to \infty} \limsup_{N \to \infty} \frac{\eta_g(N)}{\sqrt{gN}}.$$
\end{theorem}
When it appeared in 2010, Theorem~\ref{thm:sidon} was an isolated result in the literature.  Our Main Theorem shows that this discrete-continuous connection is in fact an example of a more general phenomenon in additive combinatorics and analysis.

We emphasize that it is not known in general whether or not the limits
$$\lim_{N \to \infty} \frac{\beta_g(N)}{\sqrt{N}} \quad \text{and} \quad \lim_{N \to \infty} \frac{\eta_g(N)}{\sqrt{N}}$$
exist for all values of $g$ (although it seems very likely that they do); in this sense the two above theorems are the best possible given current machinery.

We conclude this section by recording the state-of-the-art bounds for these two (pairs of) problems.  Because of the two theorems above, any improvement on either the discrete side or the continuous side immediately gives a corresponding improvement in the other area.  For the generalized Sidon set context, we know that
$$1.147\ldots \leq \sigma \leq 1.252\ldots$$ (lower bound due to Cloninger and Steinerberger \cite{cloninger}, from the continuous world; upper bound due to Matolcsi and Vinuesa \cite{matolcsi}, from the discrete world).  See also \cite{cilleruelo2, cilleruelo, green, kolou, martin}.  For the generalized difference set context, we are not aware of previous work on $\eta_g(N)$ for $g>2$.  For $g=2$, the current best bounds are (for $N$ sufficiently large)
$$\sqrt{2.435n} \leq \eta_g(N) \leq \sqrt{2.645n}$$
(lower bound due to Bernshteyn and Tait \cite{bernshteyn}; upper bound due to Golay \cite{golay}.)  The continuous analog has received more attention, and we know that
$$1.560\ldots<\tau \leq 1.643\ldots$$
(lower bound due to Madrid and Ramos \cite{madrid}; upper bound due to Barnard and Steinerberger \cite{barnard}).  See also the discussion in \cite{fish}.

\subsection{Structure of the paper}

The remainder of this paper could be summarized in a single sentence as follows: \emph{Mutatis mutandis}, the techniques of Cilleruelo, Ruzsa, and Vinuesa \cite{cilleruelo} transfer to the setting of difference sets.  But some of the \emph{mutanda} are quite delicate, and it is worth working through the details carefully.  In Section~\ref{sec:finite}, we use number-thoretic arguments from \cite{cilleruelo} to construct small $g$-difference sets in cyclic groups of certain orders; one consequence of our construction is that
$$\liminf_{N \to \infty} \frac{\gamma_g(\mathbb{Z}/N\mathbb{Z})}{\sqrt{Ng}}=1+O\left(g^{-2/5} \right)=1+o(1)$$
(where the asymptotic notation is with respect to $g$).  In Section~\ref{sec:sets-to-functions}, we prove that
$$\frac{\eta_g(N)}{\sqrt{gN}} \geq \tau$$
for all $g, N$.  In Sections~\ref{sec:functions-to-sets} and \ref{sec:everything}, we prove that this inequality is asymptotically tight for $g$ large and $N$ large relative to $g$, which completes the proof of the Main Theorem.  In Section~\ref{sec:random}, we show that almost all random subsets of sufficiently large finite abelian groups are ``good'' $g$-difference sets  (for which result even the generalized Sidon set analog was not previously known).  In Section~\ref{sec:conclusion}, we discuss some consequences of our results and raise a few open questions for future inquiry.

\section{Constructions in certain finite groups}\label{sec:finite}
The aim of this section is to show that
$$\liminf_{N \to \infty}\frac{\gamma_g(\mathbb{Z}/N\mathbb{Z})}{\sqrt{N}}=\sqrt{g}(1+o(1)).$$

Our first result says that the trivial lower bound $\gamma_g(G) \geq \sqrt{g|G|}$ is nearly sharp for $G=(\mathbb{Z}/p\mathbb{Z})^2$ and some particular values of $g$.  The argument is a modification of Theorem~3.1 of Cilleruelo, Ruzsa, and Vinuesa \cite{cilleruelo}.

\begin{theorem}
Fix any positive integer $k$.  For every sufficiently large prime $p$, there exists a $g$-difference set $A \subseteq (\mathbb{Z}/p\mathbb{Z})^2$ of size $|A|=k(p-1)+1$ with $$g= \left\lceil k^2-2(k-1)-2k^{3/2} \right\rceil=k^2+O(k^{3/2}).$$
\end{theorem}

\begin{proof}
First, for each nonzero element $u \in \mathbb{Z}/p\mathbb{Z}$, define the subset
$$A_u=\left\{ \left(x, x^2/u \right): x \in \mathbb{Z}/p\mathbb{Z} \right\} \subset (\mathbb{Z}/p\mathbb{Z})^2.$$
Given $\underline{a} \in (\mathbb{Z}/p\mathbb{Z})^2$, let $$r_{u,v}(\underline{a})=|\{(\underline{a_1}, \underline{a_2}) \in A_u \times A_v: \underline{a}=\underline{a_1}-\underline{a_2}\}|$$
denote the number of representations of $\underline{a}$ as the difference of an element of $A_u$ and an element of $A_v$.

We now claim that if $u-v=u'-v'$ and $\left( \frac{uvu'v'}{p} \right)=-1$ (i.e., $uvu'v'$ is a nonzero quadratic non-residue modulo $p$), then every $\underline{a} \in (\mathbb{Z}/p\mathbb{Z})^2$ satisfies $$r_{u,v}(\underline{a})+r_{u',v'}(\underline{a})=2.$$
To see that this is the case, fix some $\underline{a}=(a,b)$.  Then $r_{u,v}(\underline{a})$ counts the pairs $(x,y) \in (\mathbb{Z}/p\mathbb{Z})^2$ satisfying $$x-y=a \quad \text{and} \quad \frac{x^2}{u}-\frac{y^2}{v}=b.$$
Substituting $x=y+a$ into the second equation and clearing denominators gives
$$v(y+a)^2-u(y^2)=buv,$$
i.e.,
$$(v-u)y^2+(2av)y+(a^2v-buv)=0.$$
Note that we have $u \neq v$, since $u=v$ would imply $u'=v'$, whence $uvu'v'$ would be a quadratic residue.  So this equation is in fact a quadratic, and its discriminant is $$\Delta=4uv(a^2-b(u-v));$$ the number of solutions for $y$ is $$r_{u,v}(\underline{a})=1+\left( \frac{\Delta}{p} \right).$$
Likewise, $r_{u',v'}(\underline{a})=1+\left( \frac{\Delta'}{p} \right)$, where $\Delta'=4u'v'(a^2-b(u'-v'))$, and we wish to show that $\left( \frac{\Delta}{p} \right)+\left( \frac{\Delta'}{p} \right)=0$.  Using $u-v=u'-v'$ and $\left( \frac{uvu'v'}{p} \right)=-1$, we compute:
\begin{align*}
\left( \frac{\Delta}{p} \right) \left( \frac{\Delta'}{p} \right) &=\left( \frac{16}{p} \right) \left( \frac{uvu'v'}{p} \right) \left( \frac{(a^2-b(u-v))(a^2-b(u'-v'))}{p} \right)\\
 &=-\left( \frac{(a^2-b(u-v))^2}{p} \right).
\end{align*}
If $a^2-b(u-v)=0$, then $\left( \frac{\Delta}{p} \right)=\left( \frac{\Delta'}{p} \right)=0$, which gives the desired sum; otherwise, $\left\{\left( \frac{\Delta}{p} \right), \left( \frac{\Delta'}{p} \right) \right\}=\{-1,1\}$ also gives the desired sum.  This establishes the claim.

Next, we will identify a choice of $0 \leq t \leq p-k-1$ such that
$$A=\bigcup_{u=t+1}^{t+k} A_u$$
satisfies the conditions of the theorem.  Note that since the $A_u$'s are pairwise disjoint except for their common element $(0,0)$, we have
$$|A|=k(p-1)+1.$$
It is immediate that $r_A(0,0)=|A|$.  Also, every $\underline{a} \neq (0,0)$ satisfies
$$r_A(\underline{a}) \geq \left( \sum_{t+1 \leq u,v \leq t+k} r_{u,v}(\underline{a}) \right)-2(k-1),$$
where the second term corrects for over-counting differences involving $(0,0)$.

Write $u=t+i$ and $v=t+j$, where $i,j$ range from $1$ to $k$.  For each integer $-(k-1) \leq \ell \leq k-1$, consider the $k-|\ell|$ pairs $(i,j)$ with $i-j=\ell$.  Among these pairs, some $n$ of them have $\left( \frac{uv}{p} \right)=1$, and the remaining $k-|\ell|-n$ of them have $\left( \frac{uv}{p} \right)=-1$.  So we can make $\min\{n, k-|\ell|-n\}$ disjoint pairs of pairs (not to be confused with pairs of disjoint pairs) $((i,j),(i',j'))$ each with $\left( \frac{uvu'v'}{p} \right)=-1$.  Now, the claim above gives the estimate
\begin{align*}
\sum_{i-j=\ell} r_{u,v}(\underline{a}) &\geq 2 \min\{n, k-|\ell|-n\}\\
 &=(n+(k-|\ell|-n))-|n-(k-|\ell|-n)|\\
 &=(k-|\ell|)-\left| \sum_{i-j=\ell} \left( \frac{uv}{p} \right) \right|.
\end{align*}
Summing over all values of $\ell$ (and including the original error term) gives
$$r_A(\underline{a}) \geq -2(k-1)+k^2-\sum_{|\ell| \leq k-1} \left| \sum_{i-j=\ell} \left( \frac{(t+i)(t+j)}{p} \right) \right|.$$
For convenience, set
$$S_t=\sum_{|\ell| \leq k-1} \left| \sum_{i-j=\ell} \left( \frac{(t+i)(t+j)}{p} \right) \right|;$$
we will show that this is small on average.  The Cauchy-Schwarz Inequality gives:
\begin{align*}
\sum_{t=0}^{p-1}S_t &=\sum_{t=0}^{p-1} \sum_{|\ell| \leq k-1} \left| \sum_{i-j=\ell} \left( \frac{(t+i)(t+j)}{p} \right) \right|\\
 & \leq \sqrt{p(2k-1) \sum_{t,\ell} \left( \sum_{i-j=\ell} \left( \frac{(t+i)(t+j)}{p} \right) \right)^2}\\
 &\leq \sqrt{2pk \sum_{i-j=i'-j'} \sum_{t} \left( \frac{(t+i)(t+j)(t+i')(t+j')}{p} \right)}\\
 &=\sqrt{2pk \sum_{i+j'=i'+j} \sum_{t} \left( \frac{(t+i)(t+j)(t+i')(t+j')}{p} \right)}
\end{align*}
Cilleruelo, Ruzsa, and Vinuesa \cite{cilleruelo} use Weil character sum bounds to show that this quantity is at most
$$\sqrt{2p^2k^2(2k-1)+8p^{3/2}k^4},$$
whence we conclude that for every sufficiently large value of $p$, there is a choice of $0 \leq t \leq p-k-1$ such that
$$S_t<2k^{3/2}.$$
This immediately gives
$$r_A(\underline{a})>k^2-2(k-1)-2k^{3/2},$$
as desired.
\end{proof}

Next, we show how to construct $g$-difference sets in $\mathbb{Z}/p^2s\mathbb{Z}$ using sets in $(\mathbb{Z}/p\mathbb{Z})^2$.
\begin{lemma}\label{lem:prime-to-cyclic}
Let $A \subseteq (\mathbb{Z}/p\mathbb{Z})^2$ be a $g$-difference set of size $|A|=m$.  Then for every positive integer $s$, there is a $g(s-1)$-difference set $C \subseteq \mathbb{Z}/p^2s\mathbb{Z}$ of size $|C|=ms$.
\end{lemma}
\begin{proof}
We describe the elements of one such set $C$ as follows: for each element of $A$, take a representative $(a,b)$ with $0 \leq a,b \leq p-1$, and then let $C$ include all elements of the form $a+cp+bsp$, for $0 \leq c \leq s-1$.  Note that $|C|=ms$.

Now, we wish to bound $r_{C}(x)$ from below for arbitrary $x \in \mathbb{Z}/p^2s\mathbb{Z}$.  Each such $x$ is (uniquely) expressible as
$$x=a+cp+bsp,$$
where $0 \leq a,b \leq p-1$ and $0 \leq c \leq s-1$.  So we need to find solutions (modulo $p^2s$, of course) to the equation
$$a+cp+bsp=(a_1+c_1p+b_1sp)-(a_2+c_2p+b_2sp),$$
where $(a_1,b_1), (a_2,b_2) \in A$ (with $0 \leq a_1,b_1,a_2,b_2 \leq p-1$) and $0 \leq c_1,c_2 \leq s-1$.

By assumption, there exist at least $g$ solutions to $$(a_1,b_1)-(a_2,b_2)=(a,b),$$
with $(a_1,b_1), (a_2,b_2) \in A$.  For each such solution, take representatives $0 \leq a_1,b_1,a_2,b_2 \leq p-1$.  Note that either $a_1-a_2=a$ or $a_1-a_2=a-p$.  If $a_1-a_2=a$, then we have
$$(a_1+(c+c_1)p+b_1sp)-(a_2+c_1p+b_2sp)=a+cp+bsp$$
for each of the $s-c$ choices of $0 \leq c_1 \leq s-1-c$.  If $a_1-a_2=a-p$, then we have
$$(a_1+(c+c_1+1)p+b_1sp)-(a_2+c_1p+b_2sp)=a+cp+bsp$$
for each of the $s-c-1$ choices of $0 \leq c_1 \leq s-2-c$.  So over all of these solutions, we obtain at least $g(s-c-1)$ such expressions of $a+cp+bsp$ as a difference.

Likewise, the assumption on $A$ tells us that there exist at least $g$ solutions to $$(a_1,b_1)-(a_2,b_2)=(a,b+1),$$
with $(a_1,b_1), (a_2,b_2) \in A$.  As above, take representatives between $0$ and $p-1$.  If $a_1-a_2=a$, then we have
$$(a_1+(c-c_1)p+b_1sp)-(a_2+(s-c_1)p+b_2sp)=a+cp+bsp$$
for each of the $c$ choices of $1 \leq c_1 \leq c$.  If $a_1-a_2=a-p$, then we have
$$(a_1+(c-c_1+1)p+b_1sp)-(a_2+(s-c_1)p+b_2sp)=a+cp+bsp$$
for each of the $c+1$ choices of $1 \leq c_1 \leq c+1$.  So over all of these solutions, we obtain at least $gc$ such expressions of $a+cp+bsp$ as a difference.  Combining these two cases gives a total of at least $g(s-1)$ total solutions, so $r_C(x) \geq g(s-1)$.
\end{proof}

Combining the previous two results gives the following.

\begin{corollary}\label{cor:cyclic}
For all positive integers $k,s$ and for every sufficiently large prime $p$ (relative to $k$), there exists a $g$-difference set $C \subseteq \mathbb{Z}/p^2s\mathbb{Z}$ of size $|C|=s(pk-k+1)$ with $g=\left\lceil k^2-2(k-1)-2k^{3/2} \right\rceil(s-1)$.
\end{corollary}

Setting $k=4s^2$ in this corollary and applying the Prime Number Theorem gives the following asymptotic result, as in \cite{cilleruelo}.  (Recall that $\gamma_g(N) \geq \sqrt{gN}$ for all $g,N$.)

\begin{theorem}
We have
$$\liminf_{N \to \infty} \frac{\gamma_g(\mathbb{Z}/N\mathbb{Z})}{\sqrt{N}}=\sqrt{g}+O\left(g^{3/10}\right)=\sqrt{g}(1+o(1)).$$
\end{theorem}

\section{From sets to functions} \label{sec:sets-to-functions}

The ``easy direction'' of the main result of Cilleruelo, Ruzsa, and Vinuesa \cite{cilleruelo} transfers to the setting of difference sets with little difficulty.

\begin{theorem}\label{thm:easy-direction}
If $A \subset \mathbb{Z}$ is a $g$-difference set for $[N]$, then there is a function $f \in \mathcal{F}$ with $$\int_{\mathbb{R}} f(x) \, dx= \frac{|A|}{\sqrt{gN}}.$$
In particular, we have
$$\tau \leq \frac{\eta_g(N)}{\sqrt{gN}}$$
for all choices of $g,N$.
\end{theorem}

\begin{proof}
Let $A \subset \mathbb{Z}$ be a $g$-difference set for $[N]$ of size $|A|=\eta_g(N)$.  We now define a function $f \in \mathcal{F}$ via
$$f(x)=\begin{cases}
\sqrt{N/g}, &a/N \leq x< (a+1)/N \text{ for some } a \in A\\
0, &\text{otherwise}
\end{cases}.$$
Note immediately that
$$\int_{\mathbb{R}}f(x) \, dx=\frac{|A|}{\sqrt{gN}},$$
as required.  Write $\chi_A$ for the $0-1$ indicator function of the set $A$.  For $j/N \leq x \leq (j+1)/N$ (where $j \in 0,1,\ldots, N-1$), we have
$$(f \star f)(x)=\frac{N}{g} \left[ \sum_{i \in \mathbb{Z}}\chi_A(i) \chi_A(i+j)\left( \frac{j+1}{N}-x \right)+\sum_{i \in \mathbb{Z}}\chi_A(i)\chi_A(i+j+1)\left( x-\frac{j}{N} \right) \right].$$
This expression (which is linear in $x$) must achieve its minimum value at either $x=j/N$ or $x=(j+1)/N$, in which cases we have (respectively)
$$(f \star f)\left(\frac{j}{N}\right)=\frac{1}{g}\sum_{i \in \mathbb{Z}}\chi_A(i) \chi_A(i+j)=\frac{r_A(j)}{g}$$ and $$(f \star f)\left(\frac{j+1}{N}\right)=\frac{1}{g} \sum_{i \in \mathbb{Z}}\chi_A(i)\chi_A(i+j+1)=\frac{r_A(j+1)}{g}.$$
The assumption on $A$ ensures that each of these quantities is at least $1$, which completes the proof.
\end{proof}

\section{From functions to sets}\label{sec:functions-to-sets}

We now turn out attention to the ``hard direction'' of the Main Theorem, that is, moving from the continuous world to the discrete world.  Our argument, like the corresponding argument in \cite{cilleruelo}, is inspired by a result of Schinzel and Schmidt \cite{schinzel}.

Madrid and Ramos \cite{madrid} recently showed the existence of an extremizing measure $f \in \mathcal{F}$ with $\int_{\mathbb{R}}f(x) \, dx=\tau$, but it remains a difficult problem in analysis to say much more.  We use this result because it makes our proofs easier to read, but it is equally possible to work directly from the definition of $\tau$ as an infimum, at the cost of some extra $\varepsilon$'s that end up being inconsequential.  The following lemma is cleaner than the analogous statement in the setting of Sidon sets because we do not require our functions to have bounded support.

\begin{lemma}\label{lem:function-to-sequence}
For every $\varepsilon>0$ and every sufficiently large natural number $N$ (depending on $\varepsilon$), there exist nonnegative real numbers $\{a_i\}_{i \in \mathbb{Z}}$ satisfying the following three conditions:
\begin{enumerate}
\item $\sum_{i \in \mathbb{Z}}a_i\leq N\tau(1+\varepsilon)$.
\item Every $a_i \leq \left(\frac{1}{\tau N^{2/3}}\right) \sum_{i \in \mathbb{Z}}a_i$.
\item Every integer $1 \leq m \leq N$ satisfies $\sum_{i \in \mathbb{Z}}a_i a_{i+m} \geq N(1-\varepsilon)$
\end{enumerate}
\end{lemma}

\begin{proof}
Let $f \in \mathcal{F}$ be an extremizing measure with $\int_{\mathbb{R}}f(x) \, dx=\tau$ (where we know that such a $f$ exists from \cite{madrid}).  Let $L=\left\lceil (\tau/2) N^{2/3} \right\rceil$.  We define the ``local averages''
$$a_i=\frac{N}{2L}\int_{(i-L)/N}^{(i+L)/N} f(x) \, dx$$
for all $i \in \mathbb{Z}$.  We record the following facts:
\begin{enumerate}
\item We have
$$\sum_{i \in \mathbb{Z}}a_i=N \int_{\mathbb{R}} f(x) \, dx=N \tau$$
since the intervals $[(i-L)/N, (i+L)/N)]$ evenly cover $\mathbb{R}$ exactly $2L$ times.
\item For every $i \in \mathbb{Z}$, we have
$$a_i \leq \frac{N}{2L}\int_{\mathbb{R}}f(x) \, dx \leq \frac{N}{\tau N^{2/3}}(\tau)=\left(\frac{1}{\tau N^{2/3}}\right) \sum_{i \in \mathbb{Z}}a_i.$$
\item Fix a natural number $m \leq N-(2L-1)$.  We can write:
\begin{align*}
    \sum_{i \in \mathbb{Z}}a_i a_{i+m}
    &=\left( \frac{N}{2L} \right)^2 \sum_{i \in \mathbb{Z}} \int_{(i-L)/N}^{(i+L)/N} f(x) \, dx \int_{(i+m-L)/N}^{(i+m+L)/N} f(y) \, dy\\
     &=\left( \frac{N}{2L} \right)^2 \sum_{i \in \mathbb{Z}} \int_{(i-L)/N}^{(i+L)/N} \int_{(i+m-L)/N-x}^{(i+m+L)/N-x} f(x) f(x+z) \, dz \, dx.
\end{align*}
For each $i$, the region of integration is a diamond in the $xz$-plane with vertices
$$\left( \frac{i-L}{N}, \frac{m}{N} \right), \left( \frac{i-L}{N}, \frac{m+2L}{N} \right), \left( \frac{i+L}{N}, \frac{m}{N} \right), \left( \frac{i+L}{N}, \frac{m-2L}{N} \right).$$
As $i$ ranges over $\mathbb{Z}$, these diamonds range over (most of) the horizontal strip given by $(m-2L)/N \leq z \leq (m+2L)/N$, and horizontal lines with $z$-value closer to $m/N$ are covered more times.  To be precise: for each integer $1 \leq j \leq 2L-1$, the horizontal strip $(m-j)/N \leq z \leq (m+j)/N$ is completely covered at least $2L-j$ times.  So, using the fact that $\int_{\mathbb{R}}f(x)f(x+z) \, dx \geq 1$ for all $-1 \leq z \leq 1$, we get
\begin{align*}
    \sum_{i \in \mathbb{Z}}a_i a_{i+m}
    &\geq \left( \frac{N}{2L} \right)^{\! 2} \,  \sum_{j=1}^{2L-1} \frac{2}{N} (2L-j)=\left( \frac{2L-1}{2L} \right) N.
\end{align*}
\end{enumerate}

Finally, we horizontally stretch $f$ so that this bound holds for $m$ ranging from $1$ to $N$ instead of only from $1$ to $N-2L+1$; this adjustment costs us an extra factor of $(1+\varepsilon)$ in condition (1) (where $\varepsilon \to 0$ as $N \to \infty$) and cancels out in (2).
\end{proof}

We now wish to use these sequences for a probabilistic construction of small $g$-difference sets.  As such, we define
$$p_i=\frac{\tau N^{2/3} a_i}{\sum_{j\in \mathbb{Z}}a_j}$$
for each $i \in \mathbb{Z}$.  Note that $p_i \leq 1$, which guarantees that our probabilities-to-be are well-defined.  We also record the facts that
$$\sum_{i \in \mathbb{Z}} p_i=\tau N^{2/3} \quad \text{and} \quad \sum_{i \in \mathbb{Z}}p_ip_{i+m} \geq \left( \frac{1-\varepsilon}{(1+\varepsilon)^2} \right) N^{1/3}.$$

Before we prove the main result of this section, we require a Chernoff tail bound for infinite sums of independent Boolean variables.  The finite-sum version of the following lemma appears as Corollary~1.9 in the standard text of Tao and Vu \cite{tao}; we defer the easy proof (for which we could not find a reference in the literature) to the end of this section.

\begin{lemma}\label{lem:chernoff}
Let $\{p_i\}_{i \in \mathbb{N}}$ be a sequence of real numbers in $[0,1]$ with $\sum_{i=1}^\infty=\mu <\infty$.  Let $X=\sum_{i=1}^{\infty}X_i$ be a sum of independent Boolean random variables, where each $X_i$ equals $1$ with probability $p_i$ and 0 with probability $1-p_i$ (so that $X$ takes values in $\mathbb{Z}_{\geq 0} \cup \{ \infty\}$).  Then for any $\delta>0$ we have the (two-sided) tail bound
$$\mathbb{P}(|X-\mu| \geq \delta \mu) \leq 2e^{- \min \{\delta^2/4, \delta/2\}\mu}.$$
\end{lemma}

\begin{theorem}\label{thm:sequence-to-set}
Fix any small $\varepsilon>0$.  Take $\{p_i\}_{i \in \mathbb{Z}}$ for $\varepsilon$ and some $N$ as above.  Let $A \subseteq \mathbb{Z}$ be a random subset where each integer $i$ is included in $A$ independently with probability $p_i$.  Then both of the following hold w.h.p for $N$ sufficiently large:
\begin{enumerate}
\item $|A| \leq (1+\varepsilon)\tau N^{2/3}$.
\item $A$ is a $g$-difference set for $[N]$ with parameter
$$g=\left(\frac{(1-\varepsilon)^2}{(1+\varepsilon)^2}\right) N^{1/3}.$$
\end{enumerate}
In particular, for every sufficiently large $N$ and for every $k \geq \left\lfloor (1+\varepsilon)\tau N^{2/3} \right\rfloor$, there exists such a $g$-difference set for $[N]$ of size exactly $k$.
\end{theorem}

\begin{proof}
First, Lemma~\ref{lem:chernoff} tells us that
$$\mathbb{P}\left(\left||A|-\tau n^{2/3}\right| \geq \varepsilon \tau N^{2/3} \right) \leq 2e^{-\varepsilon^2 \tau N^{2/3}/4},$$
where this quantity approaches $0$ as $N \to \infty$.

Second, fix any integer $1 \leq m \leq N$.  Consider the partition
$$\mathbb{Z}=S_1 \cupdot S_2$$
where $S_1$ contains all integers with residue between $1$ and $m$ (inclusive) modulo $2m$ and $S_2$ contains all integers with residue between $m+1$ and $2m$ (inclusive) modulo $2m$.  Note that for every $x \in \mathbb{Z}$, the integers $x$ and $x+m$ are in different parts of this partition.  Write
$$r_A(m)=\sum_{x \in S_1} \chi_A(x)\chi_A(x+m)+\sum_{x \in S_2} \chi_A(x)\chi_A(x+m),$$
where each sum is (by construction) a sum of independent Boolean random variables.  (Recall that $\chi_A$ denotes the $0-1$ indicator function of the set $A$.)  For $j=1,2$, let
$$\mu_j=\mathbb{E} \left( \sum_{x \in S_j} \chi_A(x)\chi_A(x+m) \right)=\sum_{x \in S_j} p_x p_{x+m},$$
where $\mu_1+\mu_2=\mathbb{E}(r_A(m)) \geq [(1-\varepsilon)/(1+\varepsilon)^2]n^{1/3}$.

For $\varepsilon \mathbb{E}(r_A(m))/2\leq 2 \mu_j$, Lemma~\ref{lem:chernoff} gives:
\begin{align*}
\mathbb{P} \left(\left|\sum_{x \in S_j} \chi_A(x)\chi_A(x+m)-\mu_j \right|\geq \frac{\varepsilon \mathbb{E}(r_A(m))}{2}  \right) &\leq 2e^{-\frac{\varepsilon^2 \mathbb{E}(r_A(m))^2}{16 \mu_j}}\\
 &\leq 2e^{-\frac{\varepsilon^2 \mathbb{E}(r_A(m))}{16}}.
\end{align*}
For $\varepsilon \mathbb{E}(r_A(m))/2>2 \mu_j$, Lemma~\ref{lem:chernoff} gives:
\begin{align*}
\mathbb{P} \left(\left|\sum_{x \in S_j} \chi_A(x)\chi_A(x+m)-\mu_j \right|\geq \frac{\varepsilon \mathbb{E}(r_A(m))}{2}  \right) &\leq 2e^{-\frac{\varepsilon \mathbb{E}(r_A(m))}{4}}\\
 &\leq 2e^{-\frac{\varepsilon^2 \mathbb{E}(r_A(m))}{16}}.
\end{align*}
So, in any event, taking a union bound over $j=1,2$ and using the known lower bound on $\mathbb{E}(r_A(m))$ gives that
$$\mathbb{P} \left( r_A(m) \leq \left(\frac{(1-\varepsilon)^2}{(1+\varepsilon)^2}\right) N^{1/3} \right) \leq 4e^{-\frac{\varepsilon^2 (1-\varepsilon)N^{1/3}}{16(1+\varepsilon)^2}}.$$
Taking a union bound over all $1 \leq m \leq N$ gives that
$$\mathbb{P}\left(r_A(m)\leq \left(\frac{(1-\varepsilon)^2}{(1+\varepsilon)^2}\right) N^{1/3} \text{ for some $m \in [N]$} \right) \leq 4Ne^{-\frac{\varepsilon^2 (1-\varepsilon)N^{1/3}}{16(1+\varepsilon)^2}},$$
where this probability approaches $0$ as $N \to \infty$.  Combining this fact with the observation in the first paragraph of the proof establishes the result.
\end{proof}

We conclude this section by proving Lemma~\ref{lem:chernoff}.

\begin{proof}[Proof of Lemma~\ref{lem:chernoff}]
Note that $\delta^2/4 < \delta/2$ for $\delta < 2$ and $\delta/2<\delta^2/4$ for $\delta>2$.  We describe the case $\delta<2$; the argument for $\delta > 2$ is identical.  Fix some small $2>\delta>\varepsilon>0$.  Then there exists a natural number $K$ such that $\sum_{i=1}^K p_i>\mu(1-\varepsilon)$.  Define the random variables $$R_K=\sum_{i=1}^K X_i \quad \text{and} \quad T_K=\sum_{i=K+1}^\infty X_i,$$ so that $X=R_K+T_K$.  Using finite sum Chernoff bounds, we have:
\begin{align*}
\mathbb{P}\left( |R_K-\mu| \geq \delta \mu \right) &\leq \mathbb{P}\left( |R_K-\mathbb{E}(R_K)| \geq \left(\frac{\delta-\varepsilon }{1-\varepsilon}\right) \mathbb{E}(R_K) \right)\\
 & \leq 2e^{-\frac{(\delta-\varepsilon )^2\mathbb{E}(R_K)}{4(1-\varepsilon)^2}}\\
 & <2e^{-\frac{(\delta-\varepsilon )^2\mu}{4(1-\varepsilon)}}.
\end{align*}
Moreover, we have
$$\mathbb{P}(T_K \neq 0)<\varepsilon \mu.$$
Since $X=R_K+T_K$, we thus have
$$\mathbb{P}(|X-\mu|\geq \delta \mu)<2e^{-\frac{(\delta-\varepsilon )^2\mu}{4(1-\varepsilon)}}+\varepsilon \mu.$$
Finally, note that this quantity approaches $2e^{-\delta^2 \mu/4}$ as $\varepsilon \to 0$.
\end{proof}

\section{Putting everything together}\label{sec:everything}

The following lemma shows how to use $g_1$-difference sets in finite cyclic groups to ``blow up'' $g_2$-difference sets in the integers.

\begin{lemma}\label{lem:blow-up}
Let $A \subset \mathbb{Z}$ be a $g_1$-difference set for $[N]$ of size $|A|=k$, and let $C \subseteq \mathbb{Z}/q\mathbb{Z}$ be a $g_2$-difference set of size $|C|=\ell$.  Then there is a $g_1g_2$-difference set $B \subset \mathbb{Z}$ for $[qN]$ of size $k\ell$.
\end{lemma}

\begin{proof}
We describe the elements of one such set $B$: let $\overline{C}$ be the preimage of $C$ in the interval $[1,q]$ under the canonical projection map $\mathbb{Z}\to \mathbb{Z}/q\mathbb{Z}$, and set $$B=\bigcup_{a \in A, c \in \overline{C}} qa+c.$$
Note that since these elements are all distinct, we have $|B|=k\ell$.

Now, we wish to bound $r_B(x)$ from below for arbitrary $1 \leq x \leq qN$.  Each such $x$ is (uniquely) expressible as
$$x=c+qa,$$
where $1 \leq c \leq q$ and $0 \leq a \leq N-1$.  So we need to find solutions to the equation
$$c+qa=(c_1+qa_1)-(c_2+qa_2),$$
where $a_1,a_2 \in A$ and $c_1,c_2 \in \overline{C}$.  By assumption, each of the equations
$$a_1-a_2=a \quad \text{and} \quad a_1-a_2=a+1$$
has at least $g_1$ solutions $(a_1,a_2) \in A \times A$.  Similarly, there are at least $g_2$ pairs $(c_1,c_2) \in \overline{C} \times \overline{C}$ each satisfying either $$c_1-c_2=c \quad \text{or} \quad c_1=c_2=c-q.$$
In the first case, we take the $g_1$ solutions $a_1-a_2=a$ and note that
$$(c_1+qa_1)-(c_2+qa_2)=c+qa.$$
In the second case, we take the $g_1$ solutions $a_1-a_2=a+1$ and note that
$$(c_1+qa_1)-(c_2+qa_2)=(c-q)+q(a+1)=c+qa.$$
Summing over the $g_2$ pairs $(c_1,c_2)$ gives $r_B(x) \geq g_1g_2$, as desired.
\end{proof}

We are finally ready to prove the hard direction of the Main Theorem.
\begin{theorem}\label{thm:final-steps}
Fix any small $\varepsilon>0$.  Then for all sufficiently large $g$ and sufficiently large $N$ (relative to $g$), we have
$$\frac{\eta_g(N)}{\sqrt{gN}} \leq \frac{(1+\varepsilon)^4}{1-\varepsilon}\tau.$$
\end{theorem}
\begin{proof}
We begin by gathering our tools.  Note that every sufficiently large integer $g_1$ can be expressed as
$$g_1=\left\lfloor \frac{(1-\varepsilon)^2}{(1+\varepsilon)^2}n^{1/3} \right\rfloor$$
for some natural number $n$.  Let $\nu=\nu(g_1)$ denote the largest such $n$.  Then Theorem~\ref{thm:sequence-to-set} guarantees that for every sufficiently large $g_1$ (say, $g_1 \geq g_0$) there exists a $g_1$-difference set $A$ for $[\nu]$ satisfying
$$\frac{|A|}{\sqrt{g_1\nu}} \leq \frac{(1+\varepsilon)^2}{1-\varepsilon}\tau.$$

At the same time, it follows from Corollary~\ref{cor:cyclic} that there exist natural numbers $g_2,s$ such that the following holds: for every sufficiently large prime $p$ (say, $p \geq p_h$), there exists a $g_2$-difference set $C \subseteq \mathbb{Z}/p^2s\mathbb{Z}$ satisfying
$$\frac{|C|}{\sqrt{g_2p^2s}} \leq 1+\varepsilon.$$
Let $p_h,p_{h+1},\ldots$ denote the successive primes starting at $p_h$.

We now proceed to the main argument.  Fix any $g \geq g_0g_2$.  Then there is some $g_1 \geq g_0$ such that
$$(g_1-1)g_2< g \leq g_1g_2.$$
Let $\nu=\nu(g_1)$.  For every sufficiently large $N$, there is an index $i>h$ such that
$$\nu p_{i-1}^2 s < N \leq \nu p_i^2 s.$$
Now, Lemma~\ref{lem:blow-up} says that there is a $g_1g_2$-difference set $B$ for $[\nu p_i^2 s]$ with
$$|B| \leq \left(\frac{(1+\varepsilon)^2}{1-\varepsilon}\tau \sqrt{g_1 \nu} \right)\left((1+\varepsilon)\sqrt{g_2 p_i^2s} \right)=\frac{(1+\varepsilon)^3}{1-\varepsilon}\tau \sqrt{(g_1g_2)(\nu p_i^2 s)}.$$
Of course, $A$ is also a $g$-difference set for $[N]$, so we get:
\begin{align*}
\frac{\eta_g(N)}{\sqrt{gN}} &\leq \frac{(1+\varepsilon)^3}{1-\varepsilon}\tau \sqrt{\frac{g_1g_2}{g}}\sqrt{\frac{\nu p_i^2 s}{N}}\\
 &< \frac{(1+\varepsilon)^3}{1-\varepsilon}\tau \sqrt{\frac{g_1g_2}{(g_1-1)g_2}}\sqrt{\frac{\nu p_i^2 s}{\nu p_{i-1}^2 s}}\\
 &=\frac{(1+\varepsilon)^3}{1-\varepsilon}\tau \sqrt{\frac{g_1}{g_1-1}} \left(\frac{p_i}{p_{i-1}}\right).
\end{align*}
The Prime Number Theorem implies that for all large enough $g$ and $N$ (relative to $g$)--which makes $g_1$ and $p_i$ correspondingly large--we have
$$\frac{\eta_g(N)}{\sqrt{gN}} \leq \frac{(1+\varepsilon)^4}{1-\varepsilon}\tau.$$
\end{proof}

Taking $\varepsilon \to 0$ gives the desired corollary.
\begin{corollary}\label{cor:liminf}
We have
$$\limsup_{g \to \infty} \limsup_{N \to \infty} \frac{\eta_g(N)}{\sqrt{gN}} \leq \tau.$$
\end{corollary}
The Main Theorem follows from this corollary and Theorem~\ref{thm:easy-direction}.

\section{Random subsets of finite groups}\label{sec:random}

Recall that a $g$-difference set $A$ in a finite abelian group $G$ trivially has size at least $\sqrt{g|G|}$ (in fact, strictly larger than $1/2+\sqrt{g(|G|-1)}$, by more careful counting).  The aim of this section is to show that this optimal dependence is aymptotically correct for random subsets.  We begin by showing why the analogous continuous question motivates looking at uniformly random subsets.  In what follows, let $\mathbb{T}^d=(\mathbb{R}/\mathbb{Z})^d$ denote the $d$-dimensional torus

Because every finite abelian group can be written as a direct product of cyclic groups, the problem of finding small $g$-difference sets in finite groups has the following continuous analog: find a small (in $L^1$ norm) function $h:\mathbb{T}^d \to \mathbb{R}_{\geq 0}$ such that $\int_{\mathbb{T}^d}h(t)h(x+t) \, dt \geq 1$ for all $x \in \mathbb{T}^d$.  This time, finding an optimal function is easy.  Note that any such $h$ satisfies
$$\left( \int_{\mathbb{T}^d} h(x) \, dx \right)^2 =\int_{\mathbb{T}^d} \int_{\mathbb{T}^d} h(t)h(x+t) \, dt \, dx \geq \int_{\mathbb{T}^d} 1\, dx=1,$$
which implies that $\int_{\mathbb{T}^d} h(x) \, dx \geq 1$.  At the same time, the constant function $h(x)=1$ is clearly a function of size $1$ that satisfies the required autocorrelation inequality, so we conclude that the optimal size is in fact exactly $1$.

Now, consider the discrete setting.  Write $$G=\mathbb{Z}/n_1 \mathbb{Z} \times \cdots \times \mathbb{Z}/n_d\mathbb{Z},$$ where each $n_i$ divides $n_{i+1}$.  Given any $g$-difference set $A \subseteq G$, a natural generalization of the argument of Theorem~\ref{thm:easy-direction} demonstrates the existence of a function $h: \mathbb{T}^d \to \mathbb{R}_{\geq 0}$ with $L^1$ norm $|A|/\sqrt{g|G|}$ and autocorrelation integral uniformly bounded below by $1$.  This result and the observations of the previous paragraph together give another view on the trivial bound $|A| \geq \sqrt{g|G|}$.

If we want to go in the other direction (that is, from functions to sets), then we may as well start with the constant function, which we already know is optimal.  Note that $G$ embeds naturally as a discrete subgroup of $\mathbb{T}^d$.  Localized averages of constant function around these embedded points are all $1$, so we will build a random subset $A$ by including each element (independently) with probability $\sqrt{g/|G|}$.  This way, the expected size of $A$ is $\sqrt{g|G|}$, and the expected value of $r_A(m)$ is $g$ for each nonzero $m \in G$.  Finally, we can use Chernoff bounds to show that with high probability, the truth is very close to the expectation.  We make this line of reasoning rigorous in the following theorem.

\begin{theorem}\label{thm:cyclic}
Fix any small $\delta, \varepsilon >0$.  Let $g: \mathbb{N} \to \mathbb{N}$ satisfy $1 \leq g(n) \leq n$ and grow such that $n=o\left(e^{\delta^2g(n)/45} \right)$.  Given a finite abelian group $G$, define a random subset $A \subseteq G$ where each element of $G$ is included independently with probability $\sqrt{g(|G|)/|G|}$.  As $|G| \to \infty$, the set $A$ w.h.p is a $g(|G|)(1-\delta)$-difference set of size at most $\sqrt{g(|G|)|G|}(1+\varepsilon)$.
\end{theorem}

\begin{proof}
Write $g=g(|G|)$.  First, Chernoff gives that
$$\mathbb{P}\left(\left||A|-\sqrt{g|G|}\right|\geq \varepsilon \sqrt{g|G|}\right)\leq 2e^{-\varepsilon^2 \sqrt{g|G|}/4},$$
and the quantity on the right-hand side approaches $0$ as $|G| \to \infty$.

Second, fix any nonzero $m \in G$.  Consider the case where $m$ has order greater than $2$ in $G$; we will explain later how to modify the argument for the case where $m$ has order exactly $2$.  Let $H$ denote the cyclic subgroup of $G$ generated by $m$.  We now describe a tripartite partition $$G=S_1 \cupdot S_2 \cupdot S_3.$$  We distribute the elements of each coset $c+H$ as follows.  List the elements $$c, c+m, c+2m, \ldots, c+(|H|-1)m.$$  If $|H| \not\equiv 1 \pmod{3}$, then assign these elements alternately to $S_1, S_2, S_3$, i.e., $S_1$ gets $c, c+3m, \ldots$ and $S_2$ gets $c+m, c+4m, \ldots$ and $S_3$ gets $c+2m, c+5m, \ldots$.  If $|H| \equiv 1 \pmod{3}$, then we modify this assignment by moving $c$ from $S_1$ to $S_3$.  Note that the size of each $S_j$ is either $(|G|/|H|)\lfloor|H|/3 \rfloor$ or $(|G|/|H|)\lceil|H|/3 \rceil$ and that no $S_j$ contains both $x$ and $x+m$ for any $x \in G$.

Now, write
$$r_A(m)=\sum_{j=1}^3 \sum_{x \in S_j} \chi_A(x) \chi_A(x+m),$$
where each inner sum is (by construction) a sum of independent Boolean random variables, each of which equals $1$ with probability $g/|G|$.  (Recall that $\chi_A$ denotes the $0-1$ indicator function for the set $A$.)  We have $$\mathbb{E} \left( \sum_{x \in S_j} \chi_A(x) \chi_A(x+m) \right)=\frac{g|S_j|}{|G|} \geq \frac{g}{5}.$$  Chernoff bounds give that
$$\mathbb{P}\left(\left|\sum_{x \in S_j} \chi_A(x) \chi_A(x+m)-\frac{g|S_j|}{|G|}  \right| \geq  \frac{\delta}{3} \cdot \frac{g|S_j|}{|G|} \right) \leq 2e^{-\frac{\delta^2 g|S_j|}{9|G|}} \leq 2e^{-\frac{\delta^2g}{45}}.$$
A union bound then gives (recalling that $|S_1|+|S_2|+|S_3|=|G|$)
$$\mathbb{P}(|r_A(m)-g| \geq \delta g) \leq 6e^{-\frac{\delta^2g}{45}}.$$

At this point, we sketch the case where $m$ has order $2$: we use only a bipartite partition $G=T_1 \cupdot T_2$ into $2$ equal parts by making alternate assignments of the elements of each coset; carrying through the (admittedly wasteful) analogous arguments gives the (still better) bound
$$\mathbb{P}(|r_A(m)-g| \geq \delta g) \leq 4e^{-\frac{\delta^2g}{8}}.$$

Next, taking a union bound over all $m \in G$ gives that
$$\mathbb{P}(|r_A(m)-g| \geq \delta g \text{ for some $m \in G$}) \leq 6|G|e^{-\frac{\delta^2g}{45}},$$
and this probability goes to $0$ as $|G| \to \infty$ because of the condition $|G|=o \left(e^{\delta^2 g/45} \right)$.
\end{proof}
Compared to Corollary~\ref{cor:cyclic}, the finite cyclic group special case of Theorem~\ref{thm:cyclic} is stronger in the sense that it works for groups of all orders (without divisibility restrictions) and weaker in the sense that it requires $g$ to be reasonably large.

We remark that a similar argument works for $g$-Sidon sets in finite groups.

\section{Conclusion}\label{sec:conclusion}
The Main Theorem connects the constant $\tau$ (from autocorrelation integral optimzation) to the asymptotic growth of $\eta_g(N)$.  Any improved bounds in one context immediately transfer to the other.
\begin{problem}
Improve current bounds on $\tau$ and $\eta_g(N)$.  In which context is it easier to establish upper bounds, and in which is it easier to establish lower bounds?
\end{problem}

A byproduct of our method is a correspondence between the discrete and continuous worlds in the context of generalized difference sets and autocorrelation integrals: every $g$-difference set $A$ for $[N]$ gives rise to a function $f \in \mathcal{F}$ with small $L^1$ norm, and every such function $f \in \mathcal{F}$ gives rise to $g$-difference sets $A$ of small size for (suitably large) $[N]$.  Moving from $g$-difference sets to functions and then back to $g$-difference sets gains us little, but moving from functions to $g$-difference sets and than back to functions provides new insights.  In particular, for every $\varepsilon>0$ we can find some $f \in \mathcal{F}$ such that the following hold:
\begin{enumerate}
\item $\int_{\mathbb{R}} f(x) \, dx <\tau +\varepsilon$.
\item The function $f$ has bounded support.
\item There is some positive constant $c$ such that $f(x) \in \{0, c\}$ for all $x \in \mathbb{R}$.  In particular, $f$ is uniformly bounded above.
\end{enumerate}
It is not surprising that there exists such a $f \in \mathcal{F}$, but it is far from obvious how to prove this fact without recourse to the passage through the discrete world.

Madrid and Ramos \cite{madrid} demonstrated the existence of an extremizing function $f \in \mathcal{F}$ with $L^1$ norm exactly $\tau$, but their method is nonconstructive and we know very little about the structure of such a $f$.
When Barnard and Steinerberger \cite{barnard} constructed functions $f \in \mathcal{F}$ with small $L^1$ norm, their functions had integrable singularities at $x=0,1$, and they suggest that perhaps any extremizing function must have this feature.
\begin{question}
Is there a function $f \in \mathcal{F}$ with $L^1$ norm equal to $\tau$ that has bounded support and/or is uniformly bounded above?  Must every function $f \in \mathcal{F}$ with $L^1$ norm equal to $\tau$ have singularities at $x=0,1$?
\end{question}
Both parts of this question would be resolved (the first completely in the affirmative and the second in the negative) by the existence of a $g$-difference set for $[N]$ of size exactly $\tau \sqrt{gN}$.
\begin{question}
Do there exist natural numbers $g,N$ such that $\eta_g(N)=\tau\sqrt{gN}$?  (We suspect not.)
\end{question}
We also reiterate the following elusive question from the Introduction.
\begin{question}
Does the limit
$$\lim_{N \to \infty} \frac{\eta_g(N)}{\sqrt{N}}$$
exist for all natural numbers $g$?
\end{question}

We conclude with three possible generalizations of the problems discussed in this paper; one of our emphases is that the result of \cite{cilleruelo} is representative of a more widespread phenomenon.  Recall that in Section~\ref{sec:random}, we established a version of the Main Theorem for finite abelian groups, and we remarked the same program could be carried out for $g$-Sidon sets in finite abelian groups.
\begin{problem}
Find a higher-dimensional version of the Main Theorem, i.e., a connection between functions on $\mathbb{R}^n$ and $g$-difference sets in $\mathbb{Z}^n$ (with some suitable definitions).  Do the same for the main result of \cite{cilleruelo}.  Consider also the setting of $\mathbb{T}^d \times \mathbb{R}^n$ and $G \times \mathbb{Z}^n$, where $G$ is a fixed finite abelian group of order $d$.
\end{problem}
\begin{problem}
The set $A$ is a $g$-difference set for $[N]$ if the equation $m=a_i-a_j$ has at least $g$ solutions in $A$ for every $m \in [N]$.  What can we say about the case where we replace $a_i-a_j$ with a different homogeneous polynomial?  For instance, if we use the equation $5a_i-3a_j$, perhaps there is a connection to functions $f: \mathbb{R} \to \mathbb{R}_{\geq 0}$ such that $\int_{\mathbb{R}} f(5t)f(x+3t) \, dt \geq 1$ for all $x \in [0,1]$.  Similarly, the equation $m=a_i-a_j-a_k$ should correspond to the three-fold convolution $\int_{\mathbb{R}}\int_{\mathbb{R}} f(t)f(s)f(x+s+t) \, ds \, dt$.  (It would also be interesting to generalize $g$-Sidon sets in this direction.)
\end{problem}
\begin{problem}
In determining if $A$ is a $g$-difference set for $[N]$, we check if the equation $m=a_i-a_j$ has at least $g$ solutions for every $m \in [N]$.  What happens if we want $m=a_i-a_j$ to have at least, for instance, $g\sqrt{N^2-m^2}$ solutions?  On the continuous side, what can we say about functions $f: \mathbb{R} \to \mathbb{R}_{\geq 0}$ such that $(f \star f)(x) \geq \sqrt{1-x^2}$ for all $x \in [0,1]$?
\end{problem}

\section*{Acknowledgements}
The author is grateful to Stefan Steinerberger for suggesting this problem and engaging in helpful discussions throughout the writing of this paper.

\end{document}